  \documentclass[11pt]{amsart}

  \usepackage{etex} 

  \usepackage{blkarray}
  \usepackage{multirow}

  \usepackage{amsmath}
  \usepackage{amssymb}
  \usepackage{amsthm}
  \usepackage{amsfonts}
  \usepackage{booktabs}
  \usepackage[usenames,dvipsnames,svgnames]{xcolor}
  \usepackage{epsfig}
  \usepackage{fancyvrb} 
  \usepackage{hyperref}
  \hypersetup{colorlinks=true,
allcolors=blue!75!black}
  \usepackage{mathrsfs}
  \usepackage{mathtools}
  \usepackage{pictexwd, dcpic}
  \usepackage{enumitem} 

  \usepackage{graphicx}

  \usepackage{tikz}
  \usepackage{tikz-cd}
  \usepackage{caption}
  \usetikzlibrary{decorations.markings}

  \newcommand{\HH}{\mathbb{H}}

  \newcommand{\RR}{\mathbb{R}}

  \newcommand{\ZZ}{\mathbb{Z}}

  \newcommand{\vF}{\mathcal{F}}
  
  \newcommand{\vH}{\mathcal{H}}
  \newcommand{\vI}{\mathcal{I}}

  \newcommand{\vM}{\mathcal{M}}
  \newcommand{\vN}{\mathcal{N}}
  
  \newcommand{\vP}{\mathcal{P}}
  \newcommand{\vQ}{\mathcal{Q}}
  
  \newcommand{\vS}{\mathcal{S}}
  \newcommand{\vT}{\mathcal{T}}

  \newcommand{\I}{\operatorname{i}}

  \newcommand{\Aut}{\operatorname{Aut}}

  \newcommand{\Comm}{\operatorname{Comm}}

  \newcommand{\Homeo}{\operatorname{Homeo}}

  \newcommand{\Mod}{\operatorname{MCG}}

  \newcommand{\PMod}{\operatorname{PMCG}}

  \newcommand{\supp}{\operatorname{supp}}

  \newcommand{\LK}{\operatorname{Link}}
  \newcommand{\Ends}{\operatorname{Ends}}

  \newcommand{\sm}{\setminus}
  \newcommand{\ol}[1]{\overline{#1}}

  \newcommand{\abs}[1]{\left\lvert #1 \right\rvert}

  \definecolor{lightgrey}{gray}{.85}

  \theoremstyle{definition}
  \newtheorem*{defn}{Definition}
  \newtheorem*{rmk}{Remark}
  
  \newtheorem*{ex}{Example}
  \newtheorem*{exs}{Examples}

  \theoremstyle{plain}
  \newtheorem{thm}{Theorem}
  \newtheorem{lem}[thm]{Lemma}
  \newtheorem{cor}[thm]{Corollary}
  \newtheorem{prop}[thm]{Proposition}

  \newtheorem{thmmain}{Theorem}

  \theoremstyle{definition}

\begin{document}


  \title[Big Torelli groups]{Big Torelli groups: generation and commensuration}
  \author[Aramayona]{J. Aramayona}
  \address{ Universidad Aut\'onoma de Madrid \& ICMAT, Madrid, Spain}
  \email{aramayona@gmail.com}
  \author[Ghaswala]{T. Ghaswala}
  \address{Department of Mathematics, University of Manitoba, Winnipeg, Canada}
  \email{ty.ghaswala@gmail.com}
  \author[Kent]{A. E. Kent}
  \address{Department of Mathematics, University of Wisconsin, Madison, WI, United States}
  \email{kent@math.wisc.edu}
  \author[McLeay]{A. McLeay}
  \address{Mathematics Research Unit, University of Luxembourg, Esch-sur-Alzette, Luxembourg}
  \email{mcleay.math@gmail.com}
  \author[Tao]{J. Tao}
  \address{Department of Mathematics, University of Oklahoma, Norman, OK, United States}
  \email{jing@ou.edu}
  \author[Winarski]{R. R. Winarski} 
  \address{Department of Mathematics, University of Michigan, Ann Arbor, MI, United States}
  \email{rebecca.winarski@gmail.com}
  \maketitle

  \begin{abstract}
    For any surface $\Sigma$ of infinite topological type, we study the
    {\em Torelli subgroup} $\vI(\Sigma)$ of the mapping class group
    $\Mod(\Sigma)$, whose elements are those mapping classes that act
    trivially on the homology of $\Sigma$. Our first result asserts that
    $\vI(\Sigma)$ is topologically generated by the subgroup of
    $\Mod(\Sigma)$ consisting of those elements in the Torelli group which
    have compact support. In particular, using results of Birman \cite{Bi},
    Powell \cite{Po}, and Putman \cite{putman2007} we deduce that
    $\vI(\Sigma)$ is topologically generated by {\em separating twists} and
    {\em bounding pair maps}. Next, we prove the abstract commensurator
    group of $\vI(\Sigma)$ coincides with $\Mod(\Sigma)$. This extends the
    results for finite-type surfaces \cite{FI, BM, BMadden, Kida} to the
    setting of infinite-type surfaces. 

  \end{abstract}

  
\section{Introduction}

  Let $\Sigma$ be a connected oriented surface of infinite topological type
  -- that is, a surface with fundamental group that is not finitely
  generated. We will also assume the boundary components of $\Sigma$ are
  compact. The \emph{mapping class group} of $\Sigma$ is the group:
  \[\Mod(\Sigma)=\Homeo(\Sigma, \partial \Sigma)/\Homeo_0(\Sigma, \partial
  \Sigma),\] where $\Homeo(\Sigma, \partial \Sigma)$ is the group of
  self-homeomorphisms of $\Sigma$ which fix $\partial \Sigma$ pointwise,
  equipped with the compact-open topology, and  $\Homeo_0(\Sigma, \partial
  \Sigma)$ is the connected component of the identity in
  $\Homeo(\Sigma,\partial \Sigma)$. We equip $\Mod(\Sigma)$ with the
  quotient topology.  

  There is a natural homomorphism $\Mod(\Sigma) \to \Aut(H_1(\Sigma,
  \mathbb{Z}))$, whose kernel is commonly referred to as the \emph{Torelli
  group} $\vI(\Sigma)<\Mod(\Sigma)$. While Torelli groups of finite--type
  surfaces have been the object of intense study (see for example
  \cite{BBM,Bi,HM,Johnson,Kida13,MM,mess,Po,Putman12}) not much is known
  about them in the case of surfaces of infinite type. The present article
  aims to be a first step in this direction. 

  \subsection*{Generation} 

  In a recent article, Patel--Vlamis \cite{PV} give a topological
  generating set for the \emph{pure} mapping class group $\PMod(\Sigma)$
  (i.e.\ the subgroup of $\Mod(\Sigma)$ consisting of those mapping classes
  which fix every \emph{end} of $\Sigma$; see Section \ref{sec.defs}). More
  concretely, they show that $\PMod(\Sigma)$ is topologically generated by
  the subgroup of elements with compact support if $\Sigma$ has at most one
  end \emph{accumulated by genus}; otherwise, $\PMod(\Sigma)$ is
  topologically generated by the union of the set of compactly--supported
  elements and the set of \emph{handle shifts}; see Section \ref{sec.defs}. 

  Observe that $\vI(\Sigma) < \PMod(\Sigma)$. Denote by $\vI_c(\Sigma)$ the
  subgroup of $\vI(\Sigma)$ consisting of those elements with compact
  support, and let $\ol{\vI_c(\Sigma)}$ be the closure of $\vI_c(\Sigma)$
  in $\PMod(\Sigma)$. Our first result asserts that, for any infinite--type
  surface $\Sigma$, the set of compactly--supported mapping classes
  contained in the Torelli group topologically generates the Torelli group: 

  \begin{thm}
    For any connected oriented surface $\Sigma$ of infinite type, we have
    $\vI(\Sigma) = \ol{\vI_c(\Sigma)}$. 
    \label{thm.generation}
  \end{thm}

  Birman \cite{Bi} and Powell \cite{Po} showed that the Torelli group of a
  closed finite--type surface is generated by {\em separating twists}
  (i.e.\ Dehn twists about separating curves), plus {\em bounding pair
  maps} (that is, products of twists of the form $T_\gamma T_\delta^{-1}$,
  where $\gamma$ and $\delta$ are non-separating but their union
  separates). Putman then proved that the same is true for finite--type
  surfaces with boundary \cite{putman2007}.  In light of this, an immediate
  consequence of Theorem \ref{thm.generation} is: 

  \begin{cor}
    Let $\Sigma$ be a connected oriented surface of infinite topological
    type. Then $\vI(\Sigma)$ is topologically generated by separating
    twists and bounding-pair maps. 
  \end{cor}

  Theorem \ref{thm.generation} implies $\vI(\Sigma)$ is a closed subgroup
  of $\Mod(\Sigma)$.  Since $\Mod(\Sigma)$ is a Polish group \cite{APM} and
  closed subgroups of Polish groups are Polish, we have the following
  corollary.

  \begin{cor}
    Let $\Sigma$ be a connected oriented surface of infinite topological
    type. Then $\vI(\Sigma)$ is a Polish group.
  \end{cor}

  \subsection*{Commensurations} 
  Recall that, given a group $G$, its abstract commensurator $\Comm(G)$ is
  the group of equivalence classes of isomorphisms between finite-index
  subgroups of $G$; here, two isomorphisms are equivalent if they agree on
  a finite-index subgroup. Observe that there is a natural homomorphism 
  \[
  \Aut(G) \to \Comm(G).
  \] 
  We will prove: 

  \begin{thm}\label{thm.commensurator}
    For any connected oriented surface $\Sigma$ of infinite topological type and without boundary
    we have 
    \[
    \Comm \vI(\Sigma) \cong \Aut \vI(\Sigma) \cong \Mod(\Sigma).
    \] 
  \end{thm}
  
  Before continuing we stress that to the best of our knowledge, it is not known whether the Torelli group of
  an infinite--type surface has any finite-index subgroups, and so in
  particular it is possible that  $\Comm \vI(\Sigma)$ may in fact coincide
  {\em a priori} with $\Aut \vI(\Sigma)$.

  \subsection*{Historical context and idea of  proof}
 

  Theorem \ref{thm.commensurator} was previously known to hold for
  finite--type surfaces. Indeed, Farb--Ivanov \cite{FI} proved it for
  closed surfaces of genus at least five, which was then extended (and
  generalized to the Johnson Kernel) by Brendle--Margalit to all closed
  surfaces of genus at least three \cite{BM,BMadden}. Kida extended the
  result of Brendle--Margalit to all finite--type surfaces of genus at
  least four \cite{Kida}.  Finally, recent work of Brendle--Margalit and
  McLeay has further generalized the result to apply to a large class of
  normal subgroups of  finite--type surfaces \cite{BMmeta,Alan}.  
 

  In order to prove the theorem, we closely follow Brendle--Margalit's
  strategy. First, we adapt ideas of Bavard--Dowdall--Rafi \cite{BDR} to
  show that every commensuration of the Torelli group respects the property
  of being a separating twist or a bounding pair map. From this we deduce
  that every commensuration induces an automorphism of a combinatorial
  object called the \emph{Torelli complex}. This complex was originally
  introduced, for closed surfaces, by Brendle--Margalit \cite{BM}, who
  proved that its automorphism group coincides with the mapping class
  group; this was later extended by Kida \cite{Kida} to finite--type
  surfaces with punctures. Using this, plus an inductive argument due to
  Ivanov \cite{Ivanov}, we will show that every automorphism  of the
  Torelli complex of an infinite--type surface is induced by a surface
  homeomorphism. At this point, Theorem \ref{thm.commensurator} will follow
  easily using a well-known argument of Ivanov \cite{Ivanov}. 
  
  \subsection*{Acknowledgments} 
  
  This note is the result of a conversation during the conference
  ``Geometry of Teichm\"uller spaces and mapping class groups'' at the
  University of Warwick in the Spring of 2018. We are grateful to the
  organizers, Tara Brendle, Mladen Bestvina, and Chris Leininger for the
  opportunity to participate. We also thank Dan Margalit for several
  helpful comments. 
  
  We are very thankful to the referee for comments and suggestions that
  helped improve the paper, especially for an alternate argument
  establishing one of the containments in Theorem \ref{thm.generation}; see
  the end of Section \ref{section:compact.torelli}. We also thank the
  referee for telling us about Lemma \ref{Lem:Fanoni}, whose proof is due
  to Federica Fanoni. 

  Kent, Tao, and Winarski were partially supported by NSF grants
  DMS-1107452, DMS-1107263, DMS-1107367 “RNMS: Geometric Structures and
  Representation Varieties” (the GEAR Network). Aramayona was partially
  supported by grants RYC-2013-13008 and MTM-2015-67781. Kent acknowledges
  support from NSF grant DMS-1350075 and the Vilas Trustees at the
  University of Wisconsin. Tao acknowledges partial support from NSF grant
  DMS-1440140. Winarski acknowledges partial support from the AMS-Simons
  Travel Grants which are administered by the American Mathematical
  Society, with support from the Simons Foundation. 

\section{Definitions}

  \label{sec.defs}
  In this section we introduce the main objects needed for the proofs of our
  results. 

  \subsection{Surfaces} 
  
  Throughout,  by a {\em surface} we mean a connected, oriented,
  second-countable topological surface with (possibly empty) compact
  boundary. We say that $\Sigma$ has {\em finite type} if its fundamental
  group is finitely generated; otherwise, we say that $\Sigma$ has {\em
  infinite type}. In the finite type case, we will sometimes use the
  notation $\Sigma = \Sigma_{g,p}^b$, where $g$, $p$, and $b$ are,
  respectively, the genus, the number of punctures, and the number of
  boundary components of $\Sigma$. In this case, we define the {\em
  complexity} of $\Sigma$ to be the integer $\xi(\Sigma) = 3g - 3 + p+b$.

  The {\em space of ends} of $\Sigma$ is the set
  \[\Ends(\Sigma)= \varprojlim \pi_0( \Sigma\setminus K),\] 
  where the inverse limit is taken over the set of compact subsets
  $K\subset \Sigma$, directed with respect to inclusion. Here, the topology
  on $\Ends(\Sigma)$ is given by the limit topology obtained by equipping
  each $\pi_0(\Sigma\setminus K)$ with the discrete topology. See
  \cite{Richards} for further details. 

  We say that $e$ in $\Ends(\Sigma)$ is {\em accumulated by genus} if every
  neighborhood of $e$ has infinite genus; otherwise, we say that $e$ is
  {\em planar}. In particular, observe that every puncture of $\Sigma$ is a
  planar end. We denote by $\Ends_g(\Sigma)$ the subset of $\Ends(\Sigma)$
  consisting of ends accumulated by genus. It is a classical theorem (see
  \cite{Richards} for a discussion and proof) that the homeomorphism type
  of $\Sigma$ is determined by the tuple
  \[ \big( g(\Sigma), b(\Sigma), \Ends(\Sigma), \Ends_g(\Sigma)\big),\] where
  $g(\Sigma)$ and $b(\Sigma)$ denote the genus and the number of boundary
  components of $\Sigma$.
  
  \subsection{Curves and domains} 
  
  By a {\em curve} on $\Sigma$ we mean the free homotopy class of a simple
  closed curve that does not bound a disk or a disk containing a single
  planar end of $\Sigma$. Abusing notation, we will not make any
  distinction between a curve and any of its representatives. 

  We say that a curve $\gamma$ is {\em separating} if $\Sigma \setminus
  \gamma$ has two connected components; otherwise, we say that $\gamma$ is
  {\em non-separating}. We say that two curves are {\em disjoint} if they
  have disjoint representatives in $\Sigma$. A \emph{multicurve} is a set
  of pairwise disjoint curves. Given two curves $\alpha$ and $\beta$, we
  denote by $\I(\alpha,\beta)$ their geometric intersection number. The
  intersection number between two multicurves is defined additively. 

  A \emph{domain} $Y$ in $\Sigma$ is a closed subset which is itself a
  surface and the inclusion map is a proper, $\pi_1$--injective embedding.
  Note that domains are only defined up to isotopy. A \emph{subsurface} of
  $\Sigma$ is a disjoint union of domains. The following definition appears
  in \cite{BDR}. 
  
  \begin{defn} 
    A domain $Y$ of $\Sigma$ is called \emph{principal} if $Y$ has
    finite-type and every component of $\Sigma \setminus Y$ has
    infinite-type.
  \end{defn}
 
  The following lemma was communicated to us by the referee. The statement
  and the proof are due to Federica Fanoni.

  \begin{lem} \label{Lem:Fanoni}

    Let $\sigma$ be multicurve. Then there are curves realizing the
    homotopy classes in $\sigma$ which do not accumulate in any compact set
    of $\Sigma$ if and only if for every $\alpha \notin \sigma$, the set
    $\{\beta \in \sigma | \I(\alpha, \beta) \ne 0\}$ is finite.

  \end{lem}

  \begin{proof}
    
    The forward direction is clear so we will focus on the other direction.
    Choose a complete hyperbolic metric on $\Sigma$ without half planes and
    realize all curves as geodesics. The assumption on the metric implies
    that we can decompose $\Sigma$ into pairs of pants with geodesic
    boundary, funnels and/or cusps (as shown in \cite{HMVAlex}). If the
    geodesic representative of $\sigma$ has an accumulation point, then we
    can find a pair of pants $P$ of $\Sigma$ containing such a point.
    However, by assumption, there are only finitely many curves of $\sigma$
    that intersect the boundary of $P$. Thus, only finitely many curves of
    $\sigma$ intersect $P$. In particular, they cannot accumulate in $P$, a
    contradiction.\qedhere 

  \end{proof}

  Given an element $f \in \Mod(\Sigma)$, there is a canonical (possibly
  empty) multicurve $\partial f$ in $\Sigma$ for which $f(\partial f) =
  \partial f$, defined as follows. Let $\mathcal{O}(f)$ be the set of
  curves $\alpha$ such that $\{f^k(\alpha) | k \in \ZZ\}$ is finite. Then
  $\partial f$ is the set of curves in $\mathcal{O}(f)$ that are disjoint
  from all other elements of $\mathcal{O}(f)$. 
  See \cite{HT}[\textsection 2] and also
  \cite{BDR} for further details. 
  
  Given $f \in \Mod(\Sigma)$ and a subsurface $Y$, we say $Y$
  \emph{supports} $f$ or $f$ is \emph{supported} on $Y$ if $f$ can be
  realized by a homeomorphism which is the identity outside of $Y$. An element
  $f \in \Mod(\Sigma)$ is said to have \emph{compact} support if there is a
  compact subsurface that supports $f$. Similarly, $f$ has \emph{finite}
  support if there is a subsurface of finite-type that supports $f$. If $f$
  is a product of Dehn twists about a multicurve $\alpha$, then we will
  also say $\alpha$ supports $f$.
  
  The following statement follows from \cite{BDR}. The first assertion
  is Lemma 2.6, while the second follows from its proof. 

  \begin{lem} \label{Reducing}
    If $f \in \Mod(\Sigma)$ is nontrivial and has finite support, then $f$
    has infinite order and $\partial f$ is a nonempty multicurve in
    $\Sigma$. Furthermore, let $Y$ be the (finite) union of all finite-type
    domains in $\Sigma \setminus \partial f$, then $Y$ supports $f$.  
  \end{lem}

  \subsection{Pure mapping classes}The \emph{pure mapping class group}
  $\PMod(\Sigma)$ is the normal subgroup of $\Mod(\Sigma)$ whose elements
  fix every end of $\Sigma$. 

  The set of compactly--supported mapping classes $\PMod_c(\Sigma)$ forms a
  subgroup of $\PMod(\Sigma)$ which is normal in $\Mod(\Sigma)$. Note that
  if $f$ lies in $\PMod(\Sigma)$ and $f$ has finite support, then it
  necessarily has compact support. Since $\PMod_c(\Sigma)$ is a direct
  limit of pure mapping class groups of compact surfaces, a classical
  result due to Dehn and Lickorish (see \cite[Section 4]{FM}, for instance)
  implies that $\PMod_c(\Sigma)$ is generated by {\em Dehn twists}. 

  \subsection{Handle Shifts}

  For any subgroup $\Gamma < \Mod(\Sigma)$, we denote by
  $\overline{\Gamma}$ its topological closure in $\Mod(\Sigma)$.

  Patel--Vlamis introduced {\em handle shifts} and showed that handle
  shifts and Dehn twists topologically generate $\PMod(\Sigma)$ \cite{PV}.
  Subsequently, in \cite{APM}  it was shown that $\PMod(\Sigma) =
  \ol{\PMod_c(\Sigma)} \rtimes H$, where $H$ is a particular subgroup
  isomorphic to a direct product of pairwise commuting \emph{handle
  shifts}. We now recall the definition of a handle shift.

  Let $\Lambda$ be the surface obtained from $\RR \times [-1,1]$ by
  removing disks of radius $\frac 14$ centered at $(t,0)$ for $t$ in $\ZZ$
  and gluing in a torus with one boundary component, identifying the
  boundary of the torus with the boundary of the removed disk.  Let
  $\sigma:\Lambda \to \Lambda$ be the homeomorphism that shifts the handle
  at $(t,0)$ to the handle at $(t+1,0)$, and is the identity on $\RR \times
  \{-1,1\}$ (see \cite{APM} or \cite{PV} for an image of such a
  homeomorphism).  The isotopy class of $\sigma$ is called a handle shift
  of $\Lambda$.

  An element $h$ in $\Mod(\Sigma)$ is a {\it handle shift} if there exists
  a proper embedding $\iota:\Lambda \to \Sigma$ which induces an injective
  map on ends, and such that $[h] = [\delta]$ where
  $\delta\mid_{\iota(\Lambda)} = \sigma$ and $\delta$ is the identity
  outside $\iota(\Lambda)$.  As a consequence of our definition, we must
  have $\abs{\Ends_g(\Sigma)} \geq 2$; also, for each handle shift there is
  an attracting end $\epsilon_+$ and a repelling end $\epsilon_-$ in
  $\Ends_g(\Sigma)$, and they are distinct. 

  We say a handle shift $h$ with attracting end $\epsilon_+$ and repelling
  end $\epsilon_-$ is {\it dual} to a separating curve $\gamma$ if each
  component of $\Sigma\sm\gamma$ contains exactly one of $\epsilon_+$ and
  $\epsilon_-$.

  \subsection{Principal exhaustions} 
  
  We now introduce a minor modification of the notion of {\em principal
  exhaustion} from \cite{APM,BDR}. 

  \begin{defn} 
    A \emph{principal exhaustion} of $\Sigma$ is an infinite sequence of
    principal domains $\{P_1, P_2, \dots\}$ such that, for every $i\ge 1$,
    one has: 
    \begin{enumerate}
    \item $P_i \subset P_{i+1}$,
    \item \label{separating.boundaries} every component of $\partial P_i$
    is separating \item no component of $\partial P_i$ is isotopic to a
      component of $\partial P_{i+1}$, and
    \item $\Sigma = \bigcup P_i$.
    \end{enumerate}
  \end{defn}

  \begin{lem}\label{decomp}
    Let $\Sigma$ be a connected infinite--type surface and let $\{P_i\}$ be
    a principal exhaustion of $\Sigma$.  Then for all $i$, we have:
    \begin{itemize}
    \item for all $j>i$, $H_1(P_j)\cong H_1(P_i)\oplus M$ for some $M <
      H_1(P_j\setminus P_i)$
    \item $H_1(\Sigma)\cong H_1(P_i)\oplus M'$ for some $M' <
      H_1(\Sigma\setminus P_i)$
    \end{itemize}
  \end{lem}

  \begin{proof}[Proof of lemma]
    We will let $W$ be either $P_j$ or $\Sigma$ to prove both cases
    simultaneously.
   
    Let $\partial_1 P_i,  \ldots \partial_m P_i$ be the boundary components
    of $P_i$.  

    Since every component of $\Sigma - P_i$ is of infinite type, every
    component of $\overline{W - P_i}$ either contains an end of $\Sigma$ or
    a boundary component of $W$. So there is a collection of pairwise
    disjoint rays and arcs $\gamma_1,  \ldots, \gamma_m$ properly embedded
    in $\overline{W - P_i}$ such that $\gamma_k \cap \partial_k P_i$ is a
    single point for all $k$.

    By the Regular Neighborhood Theorem, we may deformation retract $W$
    along the $\gamma_k$, fixing $P_i$ throughout, to obtain a subsurface
    $\Delta$ homotopy equivalent to $W$ that contains $P_i$ and such that
    $P_i \cap \overline{\Delta - P_i}$ is a disjoint union of arcs
    $\alpha_1, \ldots, \alpha_m$, as pictured in Figure 1.
    
    \begin{figure}
    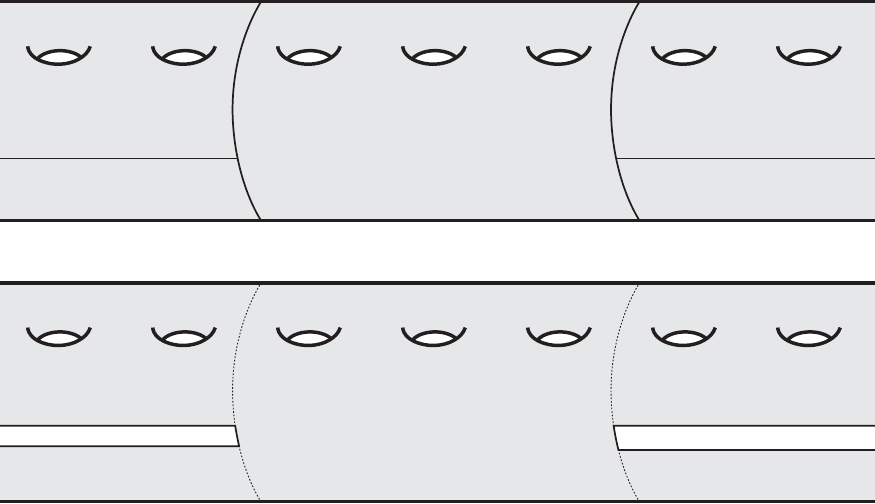
    \caption{Pictured at the top is the surface $W$, the subsurface $P_i$,
      and the arcs $\gamma_1, \ldots, \gamma_m$.  Below is the surface
      $\Delta$, obtained by deleting open neighborhoods of the interiors of
      the $\gamma_k$. The $\alpha_k$ are the dotted arcs.}
    \end{figure}

    Consideration of the Mayer--Vietoris sequence gives us an exact
    sequence 
    \[
    0 \to H_1(P_i) \oplus H_1(\Delta-P_i) \longrightarrow H_1(W)
    \mathop{\longrightarrow}^\partial H_0(\alpha_1 \sqcup \cdots \sqcup
    \alpha_m). 
    \]
    This gives us the direct sum decomposition of $H_1(W)$.  Since
    $\partial_\ell P_i$ are separating, then so are the $\alpha_\ell$. This
    implies that the boundary map $\partial$ is zero, and since $H_1(\Delta
    - P_i)$ is naturally a subgroup of $H_1(W - P_i)$, the proof is
    complete. \qedhere 
  \end{proof}

\section{Compactly generating the Torelli group}

  \label{section:compact.torelli}
  Let $\Sigma$ be an infinite--type surface.  We define the \emph{compactly
  supported Torelli group} 
  \[
  \vI_c(\Sigma) := \{f \in \vI(\Sigma)\ | f \mbox{ has compact support}\}.
  \]

  The aim of this section is to prove the first main result of the
  introduction, whose statement we now recall:

  \begin{thmmain}\label{compact_torelli}
    For any connected oriented surface $\Sigma$ of infinite type, we have
    $\vI(\Sigma) = \ol{\vI_c(\Sigma)}$. 
  \end{thmmain}
  
  \begin{rmk}
    The referee suggested an alternate way of proving that every element of
    $\vI(\Sigma)$ is a limit of mapping classes with compact support; we
    give this argument at the end of this section. However, we have decided
    to keep our original argument, phrased in terms of a new type of
    mapping class ({\em pseudo handle shifts}, see below), which may be of
    independent interest. 
  \end{rmk}


  We will need to know that certain, possibly infinite, products of handle
  shifts are inaccessible by compactly supported mapping classes.  For a
  general product of handle shifts, this is too much to hope for.  For
  example, in a surface with two ends, the product of two commuting handle
  shifts with opposite dynamics is a limit of compactly supported classes.  

  More generally, there are products of infinitely many commuting handle
  shift that are limits of compactly supported classes.  For example, there
  is the ``boundary leaf shift,'' which we now explain. 

  \begin{ex}[Boundary leaf shift]  
    
    Start with an infinite regular tree $\vT$ properly embedded in the
    hyperbolic plane $\HH^2$ with boundary a Cantor set in $\partial
    \HH^2$.  Orient $\partial \HH^2$ counterclockwise.  Build a surface by
    taking the boundary of a regular neighborhood of $\vT$ in $\HH^2 \times
    \RR$ and attach handles periodically (in the hyperbolic metric) along
    each side of $\vT$, see Figure \ref{figure.leaf}.  The orientation on
    $\partial \HH^2$ defines a product $\vH$ of handle shifts by shifting
    the handles in each region of $\HH^2 - \vT$ in the clockwise
    direction. 
    
    To see the boundary leaf shift is in $\ol{\PMod_c(\Sigma)}$, pick a
    basepoint $*$ in $\vT$ and consider the $n$--neighborhood $B(n)$ of $*$
    in $\vT$.  Then we may move the handles incident to $B(n)$ around in a
    counterclockwise fashion to get a compactly supported class $f_n$ in $
    \PMod_c(\Sigma)$.  The sequence $\{f_n\}$ converges to the boundary
    leaf shift.

    \begin{figure}[htp!]
    \includegraphics{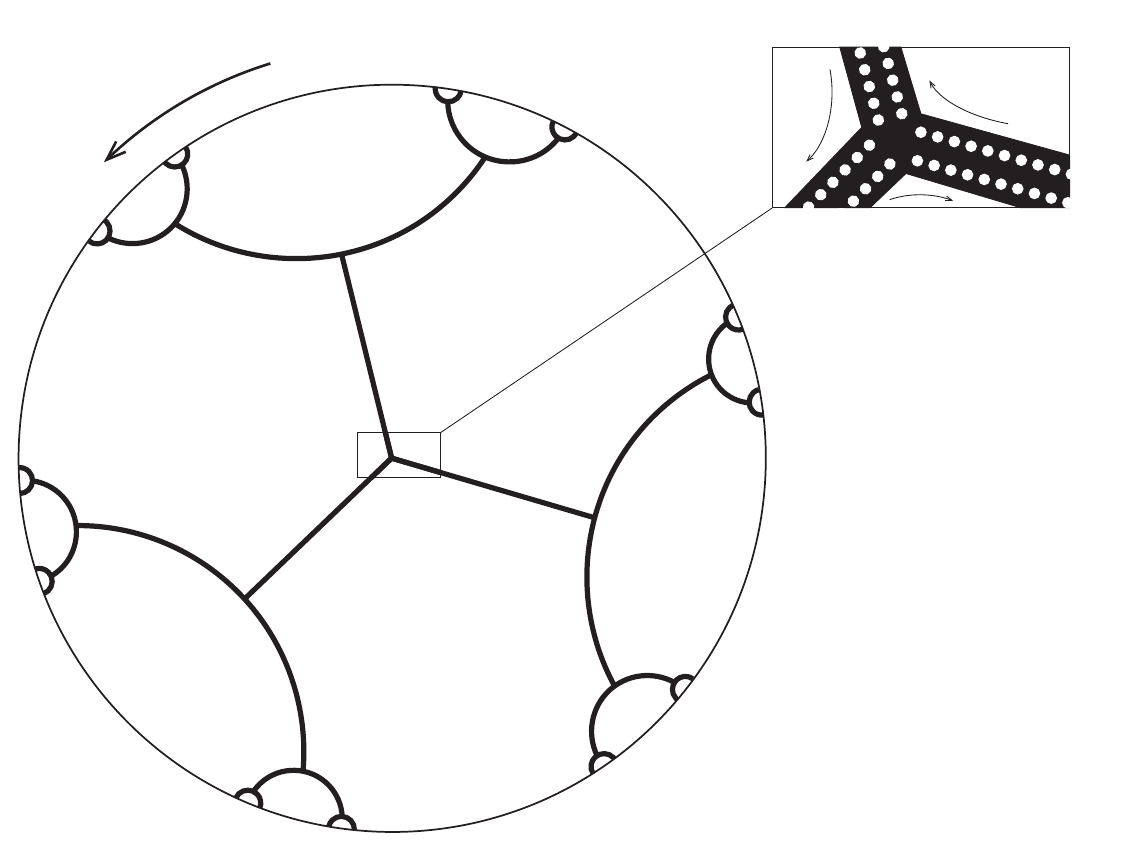}
    \caption{The boundary leaf shift.}
    \label{figure.leaf}
    \end{figure}
  \end{ex}

  Let $\gamma$ be a separating curve in $\Sigma$ whose complementary
  components are both noncompact. Let $\Sigma_-$ and $\Sigma_+$ be the
  closures of the two components of $\Sigma - \gamma$. By the same argument
  as in Lemma \ref{decomp}, $\Sigma$ deformation retracts to a subspace
  homeomorphic to $X \vee \gamma \vee Y$, where $X$ and $Y$ are subspaces
  of $\Sigma_-$ and $\Sigma_+$, respectively. It follows that $H_1(\Sigma)$
  splits as $A \oplus \langle \gamma \rangle \oplus B$, where $A = H_1(X)$
  and $B = H_1(Y)$.
  
  Similarly, if $h$ is a handle shift dual to $\gamma$, then $H_1(\Sigma)
  \cong L \oplus \langle \gamma \rangle \oplus H_1(\supp(h)) \oplus R$,
  where $L$ and $R$ are subgroups of $A$ and $B$.

  \begin{defn}[Pseudo handle shift]
    We say that a mapping class $\vH$ is a \textit{pseudo handle shift dual
    to a separating curve $\gamma$ with associated handle shift $h$} if the
    following hold:
    \begin{enumerate}
    \item $h$ is a handle shift dual to $\gamma$
    \item $\vH_*$ agrees with $h_*$ on $H_1(\supp(h))$
    \item $\vH_*([\gamma]) = [\gamma]$  
    \item $\vH_*(L) < A$ 
    \item $\vH_*(R) < B$
    \end{enumerate}
    \end{defn}
    In what follows, we always assume that the repelling end of $h$ is on
    the ``$A$--side.'' 
    
    \begin{exs} \label{ExPseudoHandleShifts}
      Let $h$ be a handle shift dual to a separating curve $\gamma$. Then
      $h$ is itself a pseudo handle shift dual to $\gamma$ (with associated
      handle shift $h$). 
      
     For a less trivial example, take a connected component $X$ of $\Sigma \setminus (\gamma \cup \supp(h))$. Then the
      composition $h$ with any element $g \in \Mod(\Sigma)$ supported on
      $X$ is a pseudo handle shift associated to $h$.
      In particular, if $h' = \prod h_i$ is a product of commuting handle
      shifts $h_i$ with dual curves $\gamma_i$, such that $h$ and $h'$
      commute, $\supp(h')$ is disjoint from $\gamma$ and $\supp(h)$ is
      disjoint from all $\gamma_i$, then the composition $k=h\circ h'$ is a
      pseudo handle shift associated to $h$. 
      
      Note that since pseudo handle shifts are
      defined by their action on homology, if $f_* = \vH_*$ for a pseudo handle shift $\vH$ 
      associated to $h$, then $f$ is also a
      pseudo handle shift associate to $h$.
    \end{exs}
    
  \begin{thm}[Pseudo handle shifts are unapproachable]
    \label{thm.pseudounapproachable}  
    A pseudo handle shift $\vH$ dual to a separating curve $\gamma$ is not
    a limit of compactly supported mapping classes.
  \end{thm}

  \begin{proof}
    Let $h$ be the associated handle shift dual to $\gamma$.  Let
    $\epsilon_-$ and $\epsilon_+$ be the ends of $\Sigma$ corresponding to
    the repelling and attracting ends of $h$, respectively, and let
    $\Sigma_-$ and $\Sigma_+$ be the complementary components of $\Sigma -
    \gamma$ containing $\epsilon_-$ and $\epsilon_+$, respectively. Choose
    some principal exhaustion $\{P_i\}$ of $\Sigma$, and let  $\Sigma_-^i
    =(\Sigma - P_i) \cap \Sigma_-$ and $\Sigma_+^i =(\Sigma - P_i) \cap
    \Sigma_+$.
    
    The curve $\gamma$ partitions the space of ends into two closed
    subspaces $E_-$ and $E_+$ of which $\Sigma_-$ and $\Sigma_+$ are
    neighborhoods, respectively.  The subsurfaces $\Sigma_-^i$ and
    $\Sigma_+^i$ are also neighborhoods of $E_-$ and $E_+$. Since $\vH$ is
    pure, $\vH(\Sigma_-^i)$ and $\vH(\Sigma_+^i)$ are also neighborhoods of
    $E_-$ and $E_+$.  Since $E_-$ and $E_+$ are disjoint, the intersection
    of the closures of $\vH(\Sigma_\pm^i)$ and $\Sigma_\mp$ is compact, and
    since the $\Sigma_\pm^i$ are nested and have empty intersection, we may
    take $i$ large enough so that the term $P_i$ in our principal
    exhaustion contains $\gamma$ and satisfies $\vH(\Sigma_\pm^i) \cap
    \Sigma_\mp$ is empty.
   
    The handle shift $h$ is supported on a strip $\vS$ with equally spaced
    handles and standard basis $\{\alpha_p, \beta_p\}_{k \in \ZZ}$ of
    $H_1(\vS)$ so that $h_*(\alpha_p) = \alpha_{p+1}$ and
    $h_*(\beta_p)=\beta_{p+1}$. We choose once and for all curves in $\vS$
    representing these classes.  After reindexing the $\alpha_p$ and
    $\beta_p$ by translating $p$, we assume that $\alpha_1$ and $\beta_1$
    lie in $\Sigma_-^i$. Since $\alpha_p$ and $\beta_p$ tend to
    $\epsilon_+$, there is some $j > 1$ such that $\alpha_j$ and $\beta_j$
    lie in $\Sigma_+^i$.

    Suppose that $\vH$ is a limit of compactly supported $\vH_n$.
    Pick $n$ large enough so that $\vH_n$ agrees with $\vH$ on $P_i$ and so
    that $\vH_{n*}$ agrees with $h_*$ on both $H_1(P_i)$ and $\langle
    \alpha_1, \beta_1, \ldots, \alpha_j, \beta_j \rangle$. Let $P_k$ be
    some term in the exhaustion with $k \geq i$ that contains the support
    of $\vH_n$.

    We have a direct sum decomposition
    \[
    H_1(P_k) \cong \ZZ^{\ell} \oplus \ZZ^{2j} \oplus \ZZ^r
    \]
    where $\ZZ^{\ell}$ is a subgroup of $H_1(\Sigma_-) \oplus \langle
    \gamma \rangle$, $\ZZ^{2j} = \langle \alpha_1, \beta_1, \ldots,
    \alpha_j, \beta_j \rangle$, and $\ZZ^r$ is a subgroup of
    $H_1(\Sigma_+)$. Picking a basis $\langle x_1, \ldots, x_{\ell},
    \alpha_1, \beta_1, \ldots, \alpha_j, \beta_j, y_1, \ldots, y_{r}
    \rangle$ for $H_1(P_k)$ compatible with this decomposition, we see that
    $\vH_{n*}$ has a block decomposition:
    \[
      \vH_{n*}=
      \begin{blockarray}{cccccc}
      &	\ell & 2j-2 & 2	& r \\
      \begin{block}{c(ccccc)}
      \ell & * & \mathbf{0} & \mathbf{0} & Y \\
      2	   & * & \mathbf{0} & \mathbf{0} & Z \\
      2j-2 & * & I	    & \mathbf{0} & * \\
      r	   & X & \mathbf{0} & A		 & B \\
      \end{block}
      \end{blockarray}	
    \]
    By properties (4) and (5) of a pseudo handle shift, and our choice of
    $i$, the blocks $X$, $Y$, and $Z$ are all zero.  So the matrix is:

    \[
    \vH_{n*}=
    \begin{blockarray}{cccccc}
    & \ell &2j-2 & 2 & r \\
    \begin{block}{c(ccccc)}
    \ell  & * & \mathbf{0} & \mathbf{0} & \mathbf{0} \\
    2	  & * & \mathbf{0} & \mathbf{0} & \mathbf{0} \\
    2j-2  & * & I	   & \mathbf{0} & *          \\
    r	  & \mathbf{0} & \mathbf{0} & A & B          \\
    \end{block}
    \end{blockarray}	
    \]
    This matrix is column equivalent to:
    \[
    \vH_{n*}=
    \begin{blockarray}{cccccc}
    & \ell & 2j-2 & 2 & r \\
    \begin{block}{c(ccccc)}
    \ell & *          & \mathbf{0} & \mathbf{0} & \mathbf{0} \\
    2	 & *          & \mathbf{0} & \mathbf{0} & \mathbf{0} \\
    2j-2 & *          & I	   & \mathbf{0} & \mathbf{0} \\
    r	 & \mathbf{0} & \mathbf{0} & A          & B          \\
    \end{block}
    \end{blockarray}	
    \]
    But the matrix $[A \ B]$ is an $r \times (r+2)$ matrix, and so its
    Jordan form cannot have a pivot in every column.  So the matrix for
    $\vH_{n*}$ is equivalent to a matrix with a zero column. But $\vH_{n*}$
    is an isomorphism, and this contradiction completes the proof. \qedhere

  \end{proof}


  We are now ready to prove Theorem \ref{compact_torelli}.

  \begin{proof}[Proof of Theorem \ref{compact_torelli}]
    We will first show that $\vI(\Sigma) < \ol{\PMod_c(\Sigma)}$. By
    \cite[Theorem 1]{PV}, we only need to consider the case when $\Sigma$
    has at least two ends accumulated by genus.  We observe that
    $\vI(\Sigma)<\PMod(\Sigma)$.  Let $g$ be in $\PMod(\Sigma)$ so that $g$
    is not a limit of compactly supported mapping classes. We show that $g$
    is not in $\vI(\Sigma)$.

    By Theorem 3 and Corollary 4 from \cite{APM}, $g$ can be written $g =
    fk^{-1}$ where $f$ is a limit of compactly supported classes and $k$ is
    a product of pairwise commuting handle shifts $h_i$. The handle shift
    $h_i$ has the property that the support of $h_i$ is disjoint from the
    dual curve $\gamma_j$ for $h_j$ whenever $i \neq j$. Thus, for any $i$,
    $k$ is a pseudo-handle shift dual to $\gamma_i$ associated to $h_i$
    (see Examples after the definition of pseudo handle shifts). If $g$
    were in Torelli, then $f_*=g_*k_* = k_*$, which implies $f$ is a pseudo
    handle shift, but this violates Theorem \ref{thm.pseudounapproachable}.
    This shows \[ \vI(\Sigma) < \ol{\PMod_c(\Sigma)}. \]

    If $\phi_n$ is a sequence in $\vI_c(\Sigma)$ that converges to $\phi$,
    then $\phi$ lies in $\vI(\Sigma)$, since $\phi_n(\alpha)$ eventually
    agrees with $\phi(\alpha)$ for any given simple closed curve $\alpha$.
    So

    \[
    \overline{\vI_c(\Sigma)} < \vI(\Sigma).
    \]

    For the other containment, let $\phi$ be an element of $\vI(\Sigma)$
    and let $\{\psi_n\}$ be a sequence in $\PMod_c(\Sigma)$ converging to
    $\phi$.  We would like to convert $\psi_n$ into a sequence of compactly
    supported $\phi_n$ in $\vI(\Sigma)$ converging to $\phi$. The idea is
    that the homology classes affected by $\psi_n$ must move further and
    further away from a given basepoint, and so we can precompose the
    $\psi_n$ with a mapping class supported far from the basepoint to
    produce the desired $\phi_n$.
    
    Fix a principal exhaustion $\{P_i\}$ of $\Sigma$. For each $i$, pick a
    $j >i$ such that $P_j$ contains $\phi^{-1}(P_i)$. Pick an $N$ large
    enough so that, for all $n \geq N$, the map $\psi_{n}$ has a
    representative that agrees with a fixed representative of $\phi$ on
    $P_j$. Note that $\psi_{n*}$ agrees with $\phi_*$ on $H_1(P_j)$. Pick a
    $k>j$ such that $P_k$ contains the support of $\psi_n$.

    By Lemma \ref{decomp}, we have $H_1(P_k) \cong H_1(P_i) \oplus Q \oplus
    R$ for some $Q$ a subgroup of $H_1(P_j-P_i)$ and $R$ a subgroup of
    $H_1(P_k-P_j)$. Let $\alpha$ be element of $H_1(P_k)$ and write $\alpha
    = \gamma + \mu + \nu$ where $\gamma$, $\mu$, and $\nu$ are in
    $H_1(P_i)$, $Q$, and $R$, respectively. So $\psi_{n*}(\alpha) = \gamma
    + \mu + \psi_{n*}(\nu)$.

    The class $\nu$ is represented by a $1$--manifold $\vN$ in $P_k - P_j$.
    By our choice of $j$ and $n$, the $1$--manifold $\psi_n(\vN)$ is
    disjoint from $P_i$. So $\psi_{n*}(\nu)$ is in $Q\oplus R$. Therefore
    $\psi_{n*} : H_1(P_k) \to H_1(P_k)$ may be represented by a square
    matrix
    \[
    A =
    \left[
    \begin{array}{cccc}
    I 	       & \mathbf{0} \\
    \mathbf{0} & B			
    \end{array}
    \right]
    \]
    where $I$ is the identity on $H_1(P_i)$ and $B$ is a square matrix.
    Since $A$ is the induced map on homology associated to a homeomorphism
    of $P_k$, it is invertible and respects the intersection form, and so
    the same is true of $B$.
    
    We claim that the matrix $B$ is represented by a homeomorphism
    $F':\overline{P_k - P_i} \to \overline{P_k - P_i}$ that is the identity
    on $\partial P_i \cap \overline{P_k -P_i}$. To see this, note that $B$
    preserves the intersection form and the homology classes of the
    boundary components of each component of $\overline{P_k - P_i}$.  Let
    $X$ be the surface obtained from $\overline{P_k - P_i}$ by capping off
    all the boundary components with disks.  The homology of $X$ is a
    quotient of that of $\overline{P_k - P_i}$, and $B$ induces an
    automorphism of the homology of $X$ that preserves the intersection
    form.  There is a homeomorphism $X \to X$ inducing this automorphism
    that preserves each component of $X$, by Burkhardt's theorem
    \cite{burkhardt}.  By an isotopy, we may assume that Burkhardt's
    homeomorphism of $X$ fixes, point-wise, small disks around the centers
    of the disks we added to construct $X$.  Restricting this to
    $\overline{P_k -P_i}$ is the desired $F'$.

    We extend $F'$ by the identity to all of $\Sigma$ and call the result
    $F$.

    Now consider the homeomorphism $\phi_n = \psi_n \circ F^{-1}$. By the
    construction of $F$, this homeomorphism acts trivially on the homology
    of $\Sigma$, and agrees with $\psi_n$ on $P_i$.

    This completes the proof. \qedhere
  \end{proof}
    
  We finish this section with the alternate proof of the fact that
  $\vI(\Sigma) < \ol{\PMod_c(\Sigma)}$, as suggested by the referee.
  Suppose, for contradiction, that there is $f\in \vI(\Sigma)$ which may
  not be expressed as the limit of a sequence of mapping classes with
  compact support. A consequence of \cite[Theorem 4.5]{APM} is that there
  is a separating curve $\gamma \subset \Sigma$ such that $\gamma$ and
  $f(\gamma)$ have different topological types in every compact subsurface
  of $\Sigma$ containing them. In particular, they induce different
  splittings of $H_1(\Sigma,\mathbb{Z})$, which contradicts the fact that
  $f$ acts trivially on homology. 

\section{Abstract commensurators of the Torelli group}

  In this section we prove Theorem \ref{thm.commensurator}.  As in the
  statement of the theorem, throughout this section we will assume that the
  surface $\Sigma$ has empty boundary.  As mentioned in the introduction,
  the first step of the argument consists of proving that an element of
  $\Comm \vI (\Sigma)$ induces a simplicial automorphism of a combinatorial
  object associated to $\Sigma$, called the {\em Torelli complex},
  introduced by  Brendle--Margalit in \cite{BM}.

  \subsection{Torelli complex} 

  Recall that the \emph{curve complex} of $\Sigma$ is the
  (infinite-dimensional) simplicial complex whose vertex set is the set of
  isotopy classes of curves in $\Sigma$, and where a collection of vertices
  spans a simplex if and only if the corresponding curves are pairwise
  disjoint.  The curve complex was used by Ivanov \cite{Ivanov}, Korkmaz
  \cite{Korkmaz}, and Luo \cite{Luo} to prove that, for all but a few
  finite--type surfaces $\Sigma$,
  \[
  \Comm \Mod(\Sigma) \cong \Aut \Mod(\Sigma) \cong \Mod(\Sigma).
  \]
  Subsequently, Bavard--Dowdall--Rafi \cite{BDR} established the analogous
  result for {\em every} infinite--type surface. In a similar fashion,
  Farb-Ivanov \cite{FI}, Brendle--Margalit \cite{BM,BMadden,BMmeta}, and
  Kida \cite{Kida} proved that, for all but a few finite--type surfaces, 
  \[ \Comm \vI(\Sigma) \cong \Aut \vI(\Sigma) \cong \Mod(\Sigma). \]

  Here, we will adapt the ideas of Brendle--Margalit \cite{BM} to the
  infinite--type setting. Given an infinite--type surface $\Sigma$, we
  define its Torelli complex to be the (infinite-dimensional) simplicial
  complex whose vertex set is the set of isotopy classes of separating
  curves and bounding pairs in $\Sigma$, and where a collection of vertices
  spans a simplex if and only if the corresponding curves are pairwise
  disjoint. In order to relax notation, we will blur the distinction
  between vertices of $\vT(\Sigma)$ and the curves (or multicurves) they
  represent.
  
  The Torelli complex of a finite-type surface is connected \cite{FI}. As a
  consequence, the same holds for the Torelli complex of an infinite-type
  surface also. We record the following observation as a separate lemma, as
  we will need to make use of it later:

  \begin{lem}
    The Torelli complex $\vT(\Sigma)$ has infinite diameter if and only if
    $\Sigma$ has finite type. \label{lem.diameter}
  \end{lem}

  \begin{proof}
    If $\Sigma$ has finite type, a slick limiting argument due to Feng Luo
    (see the comment after Proposition 4.6 of \cite{MM}) shows that the
    curve complex has infinite diameter. The obvious adaptation of this
    method to the case of the Torelli complex also implies that
    $\vT(\Sigma)$ has infinite diameter. 

    For the other direction, suppose $\Sigma$ has infinite type. Since
    curves are compact, given multicurves $\gamma, \delta \subset \Sigma$,
    we can find a separating curve $\eta \subset \Sigma$ which is disjoint
    from both $\gamma$ and $\delta$. In particular, $\vT(\Sigma)$ has
    diameter two. 
  \end{proof}

  \subsection{Automorphisms of the Torelli complex} 

  Denote by $\Aut \vT(\Sigma)$ the group of simplicial automorphisms of
  $\vT(\Sigma)$, and observe that there is a natural homomorphism
  $\Mod(\Sigma) \to \Aut \vT(\Sigma) $. We want to prove: 

  \begin{thm}\label{thm.complex}
    Let $\Sigma$ be an infinite--type surface without boundary. The natural
    homomorphism $\Mod(\Sigma) \to \Aut \vT (\Sigma)$ is an isomorphism. 
  \end{thm}

  As noted above, the finite--type case is due to Brendle--Margalit
  \cite{BM,BMadden,BMmeta} and Kida \cite{Kida}.  Indeed,  the notion of
  \emph{sides} which is used in this section is adapted from arguments that
  may be found in Brendle-Margalit \cite{BM}, and which find their way back
  to ideas of Ivanov \cite{Ivanov}.
  
  
  \subsection*{Sides}
  Recall that the {{\em link} of a vertex $v$ of a simplicial complex $X$
  is the set of all vertices of $X$ that span an edge with $v$.  In
  particular, $v$ is not an element of its link.
  For any finite-dimensional simplex $\sigma$ let $\LK(\sigma)$ be the
  intersection of the links of each of the vertices in $\sigma$.  We say
  that two vertices $\alpha, \beta$ in $\LK(\sigma)$ lie on the same
  \emph{side} of $\sigma$ if there exists a vertex $\gamma$ in $\LK(\sigma)$
  that fails to span an edge with both $\alpha$ and $\beta$, that is, if
  there exists a curve in $\LK(\sigma)$ that intersects both $\alpha$ and
  $\beta$.  Observe that ``being on the same side'' defines an equivalence
  relation $\sim_\sigma$ on $\LK(\sigma)$, that is, the \emph{sides} of
  $\sigma$ are the equivalence classes of $\sim_\sigma$ in $\LK(\sigma)$.

  In particular, we may consider the sides of a vertex of $\vT(\Sigma)$. We
  say that $\gamma$ in $\vT(\Sigma)$ is $k$-sided if there are $k$
  equivalence classes with respect to $\sim_\gamma$. As we shall see, $k$
  is in $\{1,2\}$. 

  For any vertex $\gamma$ of $\vT(\Sigma)$ there exist two subsurfaces $R,
  L \subset \Sigma$ obtained by cutting $\Sigma$ along $\gamma$ such that
  $\gamma$ is isotopic to the boundary components of both $R$ and $L$.
  Suppose $R$ is of finite type. We call $\gamma$ a \emph{pants curve} if
  $\gamma$ is a separating curve and $R \cong \Sigma_{0,2}^1$, a sphere
  with two punctures and one boundary component.  We call $\gamma$ a
  $\emph{genus curve}$ if $\gamma$ is a separating curve and $R \cong
  \Sigma_{1,0}^1$, a torus with one boundary component. If $\gamma$ is any
  other type of separating curve then we say it is type $X$.

  If $\gamma$ is a bounding pair and one of the associated subsurfaces of
  $\Sigma$ is homeomorphic to $\Sigma_{1,0}^2$ then we call it a
  \emph{genus bounding pair}.

  \begin{lem}\label{lemma.2sided}
    A vertex $\gamma$ in $\vT(\Sigma)$ is $2$-sided if and only if it is type
    $X$ or it is a genus bounding pair. Otherwise, $\gamma$ is $1$-sided. 
  \end{lem}

  \begin{proof}

    We first prove that if $\gamma$ has type $X$ then it has exactly two
    sides. Let $R$ and $L$ be the two subsurfaces of $\Sigma$ obtained by
    cutting along $\gamma$.  Let $\alpha, \beta \in \LK(\gamma)$.  If
    $\alpha \subset R$ and $\beta \subset L$, then any vertex of
    $\vT(\Sigma)$ that intersects both $\alpha$ and $\beta$ must also
    intersect $\gamma$. This implies that $\gamma$ has at least two sides.
    If $\alpha, \beta \subset R$ then there exists an element of the
    $\Mod(\Sigma)$-orbit of $\alpha$ that intersects both $\alpha$ and
    $\beta$ and is contained in $R$.  An identical argument holds for two
    vertices contained in $L$ and so it follows that $\gamma$ has exactly
    two sides.

    Now let $\gamma$ be a genus one separating curve or a pants curve.
    Define $L,R \subset \Sigma$ as above.  Observe that neither
    $\Sigma_{0,2}^1$ nor $\Sigma_{1,0}^1$ contains any non-peripheral
    separating curves or bounding pairs. Therefore $\LK(\gamma)$ does not
    contain any curves in $R$. As above, all vertices contained in $L$ are
    on the same side and so $\gamma$ is $1$-sided.

    We now move on to the case where $\gamma$ is a bounding pair.  We
    define $R$ and $L$ as above.  Assume that neither $R$ nor $L$ is
    homeomorphic to $\Sigma_{0,1}^2$  or $\Sigma_{1,0}^2$. Let $\alpha,
    \beta\in\LK(\gamma)$ be such that $\alpha \subset R$ and $\beta \subset
    L$.  As shown in Figure \ref{figure.1sidedBP}, there exists a bounding
    pair $\gamma' = \{ \delta_R, \delta_L \}$ such that:
    \begin{itemize}
    \item any pair of curves in $\gamma$ or $\gamma'$ forms a bounding pair,
    \item $\delta_R \subset R$ and $\delta_L \subset L$, and
    \item $\delta_R\cap\alpha\neq\emptyset$ and $\delta_L\cap\beta\neq\emptyset$.
    \end{itemize}
    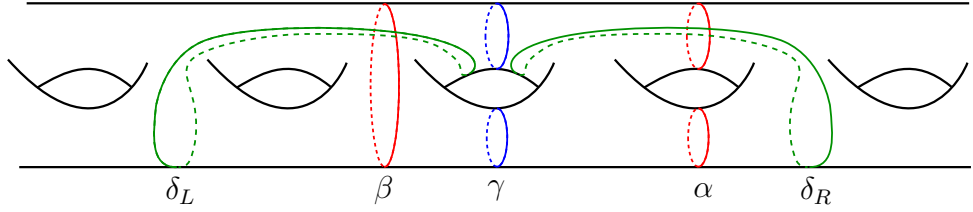
\begin{figure}[t]
    \begin{center}
    \begin{tikzpicture}
    \node[inner sep=0pt] (1sidedBP) at (0,0)
            {\includegraphics[scale=0.3]{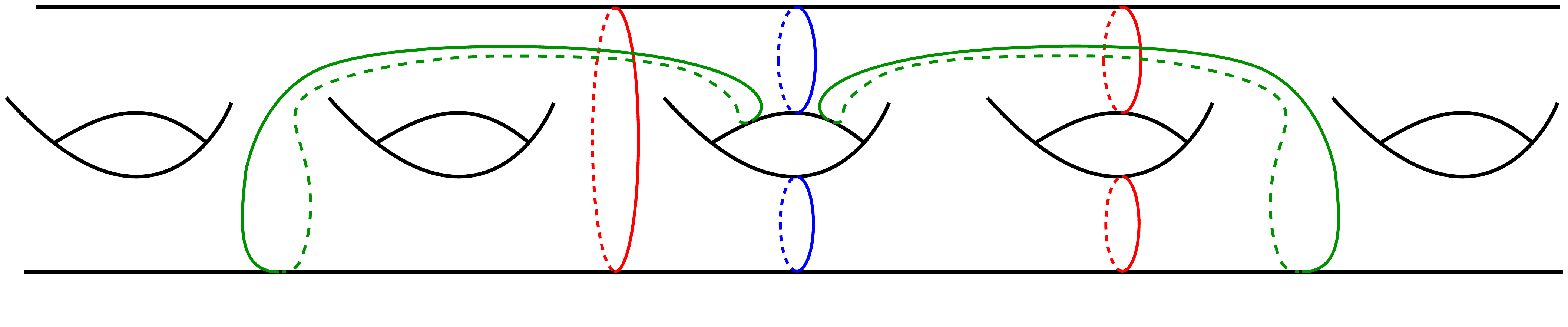}};
    \node at (0.1,-1.2) {\large $\gamma$};
    \node at (-1.4,-1.2) {\large $\beta$};
    \node at (2.85,-1.2) {\large $\alpha$};
    \node at (-4.1,-1.2) {\large $\delta_L$};
    \node at (4.35,-1.2) {\large $\delta_R$};
    \end{tikzpicture}
    \end{center}
    \caption{A general bounding pair $\gamma$ is $1$-sided. For any two
      vertices $\alpha,\beta$ in $\vT(\Sigma)$ adjacent to $\gamma$, we can
      find a bounding pair not adjacent to $\alpha$ and $\beta$ but
      adjacent to $\gamma$.  Informally, bounding pairs can ``pass
      through'' each other.}
    \label{figure.1sidedBP}
    \end{figure}

    That is, $\gamma'$ is in $\LK(\gamma)$ and there is no edge between
    $\gamma'$ and $\alpha$ or between $\gamma'$ and $\beta$.  It follows
    that $\gamma$ has exactly one side.

    If $\gamma$ is a genus bounding pair then no such $\gamma'$ exists.
    Indeed, every non-separating curve in $R$ that forms a bounding pair
    with a curve in $\gamma$ is also isotopic to a curve in $\gamma$.  By
    the same argument as for type $X$ vertices, we conclude that $\gamma$
    is $2$-sided.

    If $R$ is homeomorphic to $\Sigma_{0,1}^2$ then $\gamma$ is $1$-sided.
    Indeed, the only vertex of $\vT(\Sigma)$ contained in $R$ is $\gamma$
    and so all vertices of $\LK(\gamma)$ are contained in $L$.  This
    completes the proof. \qedhere 
  \end{proof}

  Let $\sigma$ be a finite-dimensional simplex of $\vT(\Sigma)$ consisting
  entirely of curves of type $X$. Using similar methods to the above proofs
  it is straightforward to show that the set of sides of $\sigma$ is in
  bijective correspondence with the subsurfaces of $\Sigma$ obtained by
  cutting $\Sigma$ along $\sigma$.

  The following lemma follows immediately from Lemma \ref{Lem:Fanoni}.

  \begin{lem} \label{Lem:Fanoni2}

    Let $\sigma$ be a simplex in $\mathcal{T}(\Sigma)$. Then there are
    curves realizing the homotopy classes in $\sigma$ which do not
    accumulate in any compact set of $\Sigma$ if and only if for every
    vertex $v \notin \sigma$, the set $\{w \in \sigma | \I(v, w) \ne 0\}$
    is finite.

  \end{lem}
  
  \begin{proof}[Proof of Theorem \ref{thm.complex}]
    Let \[\Phi: \Mod(\Sigma) \to \Aut(\vT(\Sigma))\] be the natural
    homomorphism; that is, for $f$ in $\Mod(\Sigma)$, $\Phi(f)$ is the
    automorphism of $\vT(\Sigma)$ determined by the rule \[\Phi(f)(\gamma)
    = f(\gamma)\] for every separating curve or bounding pair $\gamma$. 

    First, we show that $\Phi$ is injective. To this end, suppose $\Phi(f)
    = {\rm Id}$.  Then we argue that $f(\gamma) = \gamma$ for every curve
    $\gamma$. Indeed, if $\gamma$ is separating, then  $\gamma$ is a vertex
    of $\vT(\Sigma)$, so $\Phi(f)(\gamma)=\gamma$ and we are done.  If
    $\gamma$ is non-separating, there is some curve $\gamma'$ such that
    $\gamma$ and $\gamma'$ form a bounding pair.  Because $\Phi(f)$ fixes
    the vertex corresponding to $\gamma\cup\gamma'$, it must be the case
    that either $f(\gamma)=\gamma$ and $f(\gamma')=\gamma'$ or
    $f(\gamma)=\gamma'$ and $f(\gamma')=\gamma$.  But there exists a
    separating curve $\eta$ that intersects $\gamma$ but not $\gamma'$.
    Because $f(\eta)=\eta$, it cannot be the case that $f(\gamma)=
    \gamma'$.  Therefore $f(\gamma)=\gamma$ as desired.  

    By the Alexander method for infinite--type surfaces, due to
    Hern\'andez--Moralez--Valdez \cite{HMVAlex}, we deduce that $f$ is the
    identity in $\Mod(\Sigma)$.

    We now show that $\Phi$ is surjective. Let $\phi: \vT(\Sigma) \to
    \vT(\Sigma)$ be an automorphism. Fix a principal exhaustion $\{P_1,
    P_2, \dots\}$ of $\Sigma$ such that $P_1$ has complexity at least six.
    Define $\sigma_i$ to be the simplex of $\vT(\Sigma)$ corresponding to
    the multicurve $\partial P_i$.  Denote by $\vP_i$ the subcomplex of
    $\vT(\Sigma)$ spanned by the curves and bounding pairs contained in
    $P_i$.  Denote by $P_i^\circ$ the surface obtained by gluing
    once-punctured disks to each boundary component of $P_i$.  By
    construction, $\sigma_i$ contains only type $X$ vertices and therefore
    $\vT(\vP_i)$ is isomorphic to $\vT(P_i^\circ)$. As the complexity of
    $P_i$ is at least six, $\vP_i$ is connected for all $i$.  By Lemma
    \ref{lem.diameter}, we know that $\vP_i$ is the unique side of
    $\sigma_i$ whose diameter is infinite.

    Since $\phi$ is a simplicial automorphism, it induces a bijection
    between the sides of $\sigma_i$ and the sides of $\phi(\sigma_i)$.
    Because all simplicial automorphisms of $\vT(\Sigma)$ are isometries,
    $\phi(\sigma_i)$ has a unique side of infinite diameter.  From Lemma
    \ref{lemma.2sided} we have that every vertex of $\phi(\sigma_i)$ is of
    type $X$ or it is a genus bounding pair.  However, since each side of
    $\phi(\sigma_i)$ is connected, it cannot contain a genus bounding pair
    by Lemma \ref{lemma.2sided}.

    We write $\vQ_i\subset\LK(\phi(\sigma_i))$ for the unique side of
    $\phi(\sigma_i)$ with infinite diameter, and $Q_i \subset \Sigma$ for
    the finite--type subsurface which it defines. By Lemma
    \ref{Lem:Fanoni2}, we can realize the curves $\bigcup \vQ_i$ by
    non-accumulating curves in $\Sigma$. The $Q_i$'s form a sequence of
    nested subsurfaces, thus their union is an open set. The
    nonaccumulation property implies that $\bigcup Q_i$ is the
    full surface $\Sigma$.
    
    Since each vertex of $\phi(\sigma_i)$ has type $X$ we have that $\vQ_i
    \cong \vT(Q_i^\circ)$. Furthermore, $\phi$ restricts to an isomorphism
    \[
    \phi_i:\vP_i \to \vQ_i.
    \]
    Since each $P_i$ is assumed to have complexity at least six, the
    combination of results of Kida \cite{Kida} and Korkmaz \cite{Korkmaz}
    implies that $\phi_i$ is induced by a homeomorphism $f_i : P_i \to
    Q_i$.  Moreover, the homeomorphism $f_{i+1}$ may be chosen so that it
    restricts to $f_i$ on the subsurface $P_i$. Hence their direct limit is
    a homeomorphism of $\bigcup P_i$ to $\bigcup Q_i$ inducing $\phi$.
    Since $\Sigma = \bigcup P_i = \bigcup Q_i$, this completes the proof.
    \qedhere
    
  \end{proof}

  We now end this subsection with the following two observations which will
  be useful later.

  \begin{lem} \label{Rigidity1}
    For any simplex $\sigma$ in $\vT(\Sigma)$ and any compactly-supported
    $f \in \vI(\Sigma)$, if $f$ preserves $\sigma$ then $f$ fixes $\sigma$
    pointwise.
  \end{lem}

  \begin{proof}
    Consider two vertices $v, w \in \sigma$, and assume $f(v) = w$.  Then
    $v$ and $w$ are either both separating curves or they are both bounding
    pairs. 
    
    Suppose $v$ separates $\Sigma$ into two infinite-type subsurfaces.
    Since $v$ and $w$ are disjoint, we can find a compact domain $Y \subset
    \Sigma$ such that $v$ and $w$ are of different topological type. By
    choosing $Y$ large enough, we can assume that $Y$ supports $f$. But
    then it is impossible for $f|_Y(v) = f(v) = w$. 
    
    Suppose now that $v$ (and $w$) separates $\Sigma$ into a finite-type
    subsurface and an infinite-type subsurface.  Call the finite-type
    subsurfaces $V$ and $W$, corresponding to $v$ and $w$ respectively.
    Since $v\in\Sigma\setminus W$ and $w\in\Sigma\setminus V$, the
    subsurfaces $V$ and $W$ are disjoint. By assumption $f(v)=w$, therefore
    $V$ and $W$ must have the same topological type and $f(V) \subset W$.
    But this is impossible if $f \in \vI(\Sigma)$. \qedhere    

   \end{proof}   
   
   \begin{lem} \label{Rigidity2}
     For any bounding pair $v \in \vT(\Sigma)$ and any $f \in \vI(\Sigma)$,
     if $f$ preserves $v$, then $f$ must fix the curves in $v$.
   \end{lem}
     
   \begin{proof}

     Suppose $v = \{\alpha,\beta\}$ and $f(\alpha) = \beta$.  The
     complement of $v$ in $\Sigma$ has two components.  Since $f$ is pure
     it preserves the components. Let $Y$ be one of the components. The
     orientation of $Y$ induces an orientation on $\alpha$ and $\beta$.
     These oriented curves (as homology classes) satisfy $\alpha = -\beta$.
     The map $f$ is orientation preserving, so it must preserve the
     orientation of $\alpha$ and $\beta$. That is, $f(\alpha) = \beta$ as
     oriented curves. But then $\alpha$ and $f(\alpha)$ cannot be
     homologous, contradicting that $f \in \vI(\Sigma)$. \qedhere
    \end{proof}

  \subsection{Algebraic characterization of twists and bounding pair maps} 

  Before proving Theorem \ref{thm.commensurator} we will need one more
  ingredient.  Notice that the vertices of $\vT(\Sigma)$ define supports of
  elements in $\vI(\Sigma)$.  We must now show that commensurations of
  $\vI(\Sigma)$ preserve such elements and therefore define a permutation
  of the vertices of the complex.  We will adapt the algebraic
  characterization of Dehn twists of Bavard--Dowdall--Rafi \cite{BDR} to
  our setting.    

  We first introduce some terminology to facilitate the characterization of
  twists and bounding pairs. Let $G<\Mod(\Sigma)$. We denote by $\vF_G$ the
  set of elements of $G$ whose conjugacy class (in $G$) is countable.
  Bavard--Dowdall--Rafi prove that if $G$ is finite-index in $\Mod(\Sigma)$
  then $f$ is in $\vF_G$ if and only if it has compact support
  \cite[Proposition 4.2]{BDR}.  Using similar methods, we will show:
  
  \begin{prop}\label{prop.finite.support}
    Let $G < \vI(\Sigma)$ be a finite-index subgroup.  An element $f$ in $G$
    has compact support if and only if $f$ is in $\vF_G$.
  \end{prop}
  
  \begin{proof}
    It is clear that compactly-supported mapping classes have countable
    conjugacy classes. For the opposite direction, the argument in
    \cite[Proposition 4.2]{BDR} exhibits a infinite sequence of
    pairwise-disjoint curves $a_i$ such that the Dehn twists about the
    $a_i$ give rise to uncountably many conjugates of $f$. Since $\Sigma$
    has infinite type, the curves $a_i$ may be chosen to be separating, so
    that the corresponding twists belong to $\vI(\Sigma)$. Hence the result
    follows. \qedhere
  \end{proof}

  We denote by $Z(H)$ the center of $H$. If $h$ is in $H$, we write
  $C_H(h)$ for the centralizer of $h$ in $H$.  Given a finite-index
  subgroup $G < \vI(\Sigma)$ we write $\vM_G$ for the set of elements $f$
  in $G$ which
  satisfy the following three conditions: 
  \begin{enumerate}
  \item $f\in \vF_G$,
  \item $Z(\vF_G \cap C_G(f))$ is an infinite cyclic group, and
  \item $C_G(f) = C_G(f^k)$ for every $k>0$. 
  \end{enumerate}
  
  We now prove that, for any finite-index subgroup $G$ of $\vI(\Sigma)$,
  powers of Dehn twists and bounding pair maps belong to the set $\vM_G$.
  
  \begin{lem}\label{lem.multitwist}
    Let $G < \vI(\Sigma)$ be a finite-index subgroup. If $f \in G$ is a
    power of a Dehn twist about a separating curve or a bounding pair map
    then $f$ belongs to $\vM_G$.
  \end{lem}
  
  \begin{proof}
    Since $f$ has compact support, $f \in \vF_G$. Suppose first $f$ is a
    power of a Dehn twist about the separating curve $\gamma$.  We have
    that
    \[
      C_{\vI(\Sigma)}(T_\gamma^k) = \{ g \in \vI(\Sigma) \mid g ( \gamma )
      = \gamma \},
    \] 
    for $k \neq 0$.  It follows that all powers of $T_\gamma$ have the same
    centralizer in $\vI(\Sigma)$ and hence, in any subgroup.  A similar
    argument holds if $f$ is a power of bounding pair map.  This implies
    the third condition in the definition of $\vM_G$.
  
    To see that $f$ satisfies the second condition, once again assume first
    that $f$ is a power of the Dehn twist about a separating curve
    $\gamma$.  Let $g$ be a nontrivial element of $\vF_G \cap C_G(f)$ and
    assume that $g$ is not a power of $T_\gamma$.  Then there exists a
    curve $\delta$ disjoint from $\gamma$ such that $g(\delta) \neq
    \delta$.  If  $\delta$ is a separating curve then $T_\delta^k$ is in
    $C_G(f)$, for some $k>0$, but $g T_\delta^k \neq T_\delta^k g$, so $g$
    is not in $Z(\vF_G \cap C_G(f))$.  On the other hand, suppose $\delta$
    is a non-separating curve; since $g$ acts trivially on homology we have
    that $\delta$ and $g(\delta)$ are homologous curves.  Because $T_\gamma
    g(\delta)=g T_\gamma(\delta)=g(\delta)$, we have that $g(\delta)$ is
    disjoint from $\gamma$.  It follows that some power of $T_\delta
    T_{g(\delta)}^{-1}$ belongs to $C_G(f)$, but does not commute with $g$.
    We have therefore shown that $g$ is not central in $\vF_G \cap C_G(f)$
    as desired.
  
    A similar argument also shows that, if $f$ is a power of a bounding
    pair map, then $g$ belongs to $Z(\vF_G \cap C_G(f))$ if and only if it
    is a power of the same bounding pair map. \qedhere 
 
  \end{proof}
  
  When $G$ is a finite-index subgroup of $\Mod(\Sigma)$, all elements of
  $\vM_G$ are powers of multi Dehn twists, see \cite[Lemma 4.5]{BDR}.  In
  stark contrast, this is no longer true in our setting. When $G$ is a
  finite-index subgroup of $\vI(\Sigma)$, then $\vM_G$ may contain elements
  which are not supported on a disjoint union of annuli: for example, we
  may take a pure braid on a nonseparating planar subsurface with at least
  three boundary components. The following proposition characterizes those
  elements that lie in $\vM_G$, when $G$ is a finite-index subgroup of
  $\vI(\Sigma)$. Recall the definition of $\partial f$ for a mapping class
  $f$, which is non-empty when $f$ is compactly-supported. We have the
  following statements.

  \begin{lem} \label{EssSupp}
    Let $G < \vI(\Sigma)$ be a finite-index subgroup. Given a nontrival $f
    \in \vM_G$, let $\sigma \subset \partial f$ be the set of all
    separating curves or bounding pairs. The following statements hold. 
    \begin{enumerate}
      \item For all $g \in \vF_G \cap C_G(f)$, $g$ fixes every component of
        $\Sigma \setminus \partial f$.
      \item If $\sigma$ is empty, then for each finite-type domain $Y$ of
        $\Sigma \setminus \partial f$, $f|_Y$ is either the identity map or
        pseudo-Anosov. Furthermore, there is at least one pseudo-Anosov
        component of $f$. 
      \item If $\sigma$ is non-empty then $\sigma=\partial f$ is either a
        separating or a bounding pair, and $f=T_\sigma^j$ for some
        non-zero $j$, where $T_\sigma$ is either a separating twist or a
        bounding pair map.  
    \end{enumerate}
  \end{lem}

  \begin{proof}

    For any $g \in \vF_G \cap C_G(f)$, $g(\partial f) = \partial (gfg^{-1})
    = \partial f$. To show that $g$ fixes every component of $\Sigma
    \setminus \partial f$, it is enough to show $g$ fixes each component of
    $\partial f$. To see this, suppose there exists a domain $Y$of $\Sigma
    \setminus \partial f$ such that $g(Y)$ is not $Y$. Since $g$ fixes the
    components of $\partial f$, $Y$ and $g(Y)$ share the same boundary.
    This means $\Sigma = Y \cup g(Y)\cup\partial Y$, so $Y$ must have
    infinite type. But then $g$ cannot be a pure mapping class,
    contradicting the fact that $g \in \vI(\Sigma)$. In the following, we
    will consider the two cases that $\sigma$ is empty or $\sigma$ is not
    empty.  For each case, we show $g$ fixes the components of $\partial f$
    along with the statements of (2) and (3).
    
    First suppose that $\sigma$ is empty. In this case, every curve of
    $\partial f$ is non-separating and no two form a bounding pair. In
    particular, no two curves of $\partial f$ are homologous, so every
    element of $\vF_G \cap C_G(f)$ must fix the components of $\partial f$.
    Now let $Y$ be a component of $\Sigma \setminus \partial f$ of
    finite-type. The map $f|_Y$ is irreducible (otherwise the reducing
    curves would also be in $\partial f$), so it is either of finite order
    or pseudo-Anosov. We now show it is not possible for $f|_Y$ to be
    non-trivial and have finite order. Since $\partial f$ does not contain
    any separating curves or bounding pairs, it has cardinality at least
    $3$. This means $\chi(Y) < 0$, in which case $\Mod(Y)$ is torsion-free
    (\cite[Corollary 7.3]{FM}). Hence, if $f|_Y$ is not pseudo-Anosov, then
    $f|_Y$ is identity. Finally, if there is no pseudo-Anosov component of
    the support of $f$, then there exists a power $k \ge 1$ such that
    $f^k|_Y$ is identity for any component $Y$ in $\Sigma \setminus
    \partial f$. But $f$ has infinite order by Lemma \ref{Reducing}, so
    $f^k$ must be a product of powers of Dehn twists about curves in
    $\partial f$. This is impossible as $f \in \vI(\Sigma)$ \cite{Vau}. 
   
    Now suppose $\sigma$ is non-empty. Regard $\sigma$ as a simplex in
    $\vT(\Sigma)$. Every element $g \in \vF_G \cap C_G(f)$ preserves
    $\partial f$, and hence also $\sigma$. By Lemma \ref{Rigidity1}, $g$
    fixes each vertex of $\sigma$. Then, by Lemma \ref{Rigidity2}, we can
    further conclude that $g$ also fixes each curve in $\sigma$. Let $v$ by
    a vertex of $\sigma$, and let $T_v$ be the twist about $v$, which
    belongs to    $\vI(\Sigma)$. 
    
    We now show $f = T_v^j$ for some non-zero $j$. In particular,
    $v=\sigma=\partial f$. Since every $g \in \vF_G \cap C_G(f)$ preserves
    $v$, $g$ commutes with $T_v$, showing $T_v^k \in Z (\vF_G \cap C_G(f))$
    for some $k \ge 1$. Since $f$ also lies in $Z(\vF_G \cap C_G(f))$ which
    is infinite cyclic by assumption, we must have $f^m = T_v^n$ for some
    non-zero integers $m$ and $n$. Let $w$ a vertex of $\vT(\Sigma)$
    disjoint from $v$ and choose $k \ge 1$ so that $T_w^k \in G$. Since $v$
    and $w$ are disjoint, we have 
    \[ T_w^k \in C_G(T_v^n) = C_G(f^m) = C_G(f). \] This is only possible
    if $f(w) = w$. Thus, $f$ fixes every separating curve and bounding pair
    disjoint from $v$. Applying Lemma \ref{Rigidity2}, we obtain that $f$
    fixes every curve disjoint from $v$. This shows $f=T_v^j$ for some
    non-zero integer $j$, concluding the proof. \qedhere
    
  \end{proof}

  \begin{defn}[Essential Support] \label{Defn:EssSupp}
    For each element $f \in \vM_G$, we define the \emph{essential support}
    $\supp(f)$ of $f$ as follows. Let $\sigma \subset \partial f$ be the
    set of all separating curves or bounding pairs. 
    \begin{itemize}
      \item If $\sigma$ is nonempty, then $\supp(f) = \sigma$.  
      \item If $\sigma$ is empty, then $\supp(f)$ is the union over all
        domains $Y$ in $\Sigma \setminus \partial(f)$ such that $f|_Y$ is a
        pseudo-Anosov. Note that by Lemma \ref{EssSupp}, $\supp(f)$ is
        nonempty.
    \end{itemize}
  \end{defn}
  
  Let \[ C_{\vM_G}(f) = \{ g \in \vM_G \ |\  fg = gf \}. \] We also define
  a further subset: \[ (\vP_G)_f = \{ g \in C_{\vM_G}(f)  \mid g \mbox{ is
  supported in }\Sigma \setminus \supp(f)\}.
  \] The next lemma tells us that the elements supported in $\supp(f)$ are
  precisely those that are central.

  \begin{lem}\label{lem.sum}
    Let $G < \vI(\Sigma)$ be a finite-index subgroup.  For any element $f$
    in $\vM_G$ we have that 
    \[ C_{\vM_G}(f) = Z(C_{\vM_G}(f)) \oplus (\vP_G)_f. \]
  \end{lem} 
  
  \begin{proof}

    Let $\sigma \subset \partial f$ be the set of all separating curves or
    bounding pairs. If $\sigma$ is nonempty, then by Lemma \ref{EssSupp},
    $\sigma=\partial f$ is a separating curve or a bounding pair, and $f =
    T_\sigma^j$, where $T_\sigma$ is a separating twist or a bounding pair
    map. In this case, any element $g \in C_{\vM_G}(f)$ fixes every curve
    of $\sigma$. Thus, if $g$ has support in $\supp(f)$, then $g$ is itself
    a power of $T_\sigma$. This shows $g \in Z(C_{\vM_G}(f))$.
    
    Now suppose $\sigma$ is empty. By Lemma \ref{EssSupp}, $\supp(f)$
    is non-empty. Let $g \in C_{\vM_G}(f)\subset \vF\cap C_G(f)$ be an
    element with support in $\supp(f)$. We want to show $g$ lies in the
    center of $C_{\vM_G}(f)$. If $h \in C_{\vM_G}(f)$ has support disjoint
    from $\supp(f)$, then $h$ and $g$ clearly commute. Henceforth, we may
    assume this is not the case. 

    By Lemma \ref{EssSupp}, $f$, $g$, and $h$ all preserve every component
    of $\supp(f)$. Thus, we can assume without a loss of
    generality that $h$ is also supported in $\supp(f)$. Let
    $Y_1,\ldots,Y_n$ be the components of $\supp(f)$. For each $i$, let
    $f_i$, (resp.\ $g_i$ and $h_i$) denote the restriction of $f$ (resp.
    $g$ and $h$) to $Y_i$. Since the components of $\partial f$ are
    non-separating and no two form a bounding pair, we can write 
    \[
      f = f_1 f_2 \cdots f_n \qquad 
      g = g_1 g_2 \cdots g_n, \qquad 
      h = h_1 h_2 \cdots h_n,
    \]
    where each $f_i$ is pseudo-Anosov on $Y_i$. Since $\Mod(Y_i)$ is
    torsion-free and $f\in\vM_G$, the centralizer of $f_i$ in $\Mod(Y_i)$
    is cyclic. As $f$ commutes with $g$ and $h$, each $f_i$ commutes with
    $g_i$ and $h_i$. Thus, each $g_i$ or $h_i$ is contained in the
    centralizer of $f_i$, yielding $g_i$ and $h_i$ commute for all $i$.
    This shows $g$ and $h$ commute as required. \qedhere 

  \end{proof}
  
  Finally, we can prove the characterization of Dehn twists and bounding
  pair maps.

  \begin{prop}\label{cor.algcar} 
    
    Let $G <\vI(\Sigma)$ be a finite-index subgroup, and let $f$ lie in
    $G$. Then $f$ is a power of a Dehn twist or of a bounding pair map if
    and only $f$ is in $\vM_G$, and for all $g$ in $\vM_G$ such that
    $(\vP_G)_g = (\vP_G)_f$ we have that there exist integers $i,j \neq 0$
    such that $f^i=g^j$. 
  
  \end{prop}
  
  \begin{proof} 
    The forward direction is Lemma \ref{lem.multitwist} and the definition
    of $(\vP_G)_f$.
   
    For the other direction, we prove the contrapositive. Assume that $f$
    is not a power of a Dehn twist or bounding pair map. Our goal is to
    find an element $g \in \vM_G$ with $\supp(g) = \supp(f)$ such that no
    powers of $f$ and $g$ are equal, but $(\vP_G)_f=(\vP_G)_g$.
 
    By Lemma \ref{EssSupp}, there exists a domain $Y$ in $\Sigma
    \setminus \partial f$ on which $f|_Y$ is a pseudo-Anosov. Since $Y$
    supports a pseudo-Anosov, we may choose an element $h \in \Mod(\Sigma)$
    such that $h$ preserves the components of $\Sigma \setminus \partial
    f$, is identity outside of $Y$, and the restriction to $Y$ of $f$ and
    $g=hfh^{-1}$ are two independent pseudo-Anosovs. Since $\vI(\Sigma)$ is
    normal in $\Mod(\Sigma)$, $g \in \vI(\Sigma)$, and by construction,
    $\supp(g) = \supp(f)$. Thus $(\vP_G)_g = (\vP_G)_f$, but $f$ and $g$ do
    not have a common power. \qedhere
  \end{proof}

  \subsection{Abstract commensurators of the Torelli group} 

  We can now finally prove Theorem \ref{thm.commensurator}.  For a bounding
  pair $\gamma = \{ \gamma_1, \gamma_2 \}$ we use the shorthand $T_\gamma$
  for the bounding pair map $T_{\gamma_1}T_{\gamma_2}^{-1}$.

  \begin{proof}[Proof of Theorem \ref{thm.commensurator}]
    Let $[\psi]$ be an element of $\Comm \vI (\Sigma)$ representing the
    isomorphism of finite index subgroups
    \[
    \psi: G_1 \to G_2.
    \]
    Let $\gamma$ be a separating curve or a bounding pair and choose $n$ in
    $\mathbb{N}$ so that $T_\gamma^n$ is in $G_1$.  By Proposition
    \ref{cor.algcar}, $T_\gamma^n$ is in $\vM_{G_1}$ and for all $g$ in
    $\vM_{G_1}$ such that $(\vP_{G_1})_g=(\vP_{G_1})_{T_\gamma^n}$, there
    exist integers $i,j$ so  that $(T_\gamma^n)^i=g^j$.  Since these
    conditions are preserved by isomorphism, we have that
    $\psi(T_\gamma^n)$ lies in $\vM_{G_2}$,
    Proposition \ref{cor.algcar} implies there exists a separating curve or
    bounding pair $\delta$ and a nonzero integer $m$ such that
    $\psi(T_\gamma^n) = T_\delta^m.$

    At this point, and again with respect to the above notation, we obtain
    that $\psi$ induces a map 
    \begin{align*}
    \psi_* : \vT(\Sigma) &\to \vT(\Sigma);
    \\ \gamma &\mapsto \delta.
    \end{align*}
    We observe that $\psi_*$ is a simplicial map, since powers of Dehn
    twists and bounding pair maps commute if and only if the underlying
    curves are disjoint.  Moreover, the map is also bijective, with inverse
    the simplicial map associated to the inverse of $\psi^{-1}$.

    By Theorem \ref{thm.complex}, there exists an $f\in\Mod(\Sigma)$ such
    that $\psi_*(\gamma) = f(\gamma)$ for every separating curve or
    bounding pair $\gamma$. Now, for any $g$ in $G_1$ we have
    \[
    \psi (g T^n_\gamma g^{-1}) = \psi (g) \psi( T^n_\gamma) \psi( g^{-1}) =
    \psi(g) T^n_{f(\gamma)} \psi(g^{-1}) = T^n_{\psi(g)f(\gamma)},
    \]
    and therefore
    \[
    T^n_{\psi(g)f(\gamma)} = \psi(g T^n_\gamma g^{-1}) =
    \psi(T^n_{g(\gamma)}) = T^n_{fg(\gamma)}.
    \]
    Therefore $\psi(g)f(\gamma) = fg(\gamma)$.  By use of the Alexander
    method \cite{HMVAlex} we conclude that $\psi(g) = fgf^{-1}$.  This
    shows that every abstract commensurator of $\vI(\Sigma)$ is defined by
    conjugation by a mapping class of $\Sigma$, and in particular, so is
    every automorphism of $\vI(\Sigma)$. 

    On the other hand, suppose there exists an $f$ in $\Mod(\Sigma)$ and a
    finite-index subgroup $H < \vI(\Sigma)$ such that conjugation by $f$
    induces the identity map on $H$. For any separating curve or bounding
    pair $\gamma$, there exists some $m \ge 1$ such that $T_\gamma^m$ lies in
    $H$. Thus \[ T_\gamma^m = fT_\gamma^m f^{-1} = T_{f \gamma}^m. \] By
    \cite[Lemma 2.5]{BDR}, $f \gamma = \gamma$, and thus $f$ is the
    identity by Theorem \ref{thm.complex}. This completes the proof.
    \qedhere 
  \end{proof}


  \bibliographystyle{plain}
  \bibliography{BigTorelli}

\end{document}